\documentclass{amsart}

\usepackage{amsfonts,amssymb,amsmath,amsthm}
\usepackage{url}
\usepackage{enumerate}

\newtheorem{theorem}{Theorem}[section]
\newtheorem{lemma}[theorem]{Lemma}

\theoremstyle{definition}

\numberwithin{equation}{section}

\author{Redmond, Timothy 
  \and
  Ryavec, Charles
  }

\keywords{Lambda, modular, Russell type equation}
\subjclass{Primary 99X99, Secondary 99Y99}

\begin{document}
\bibliographystyle{plain}

\title[Exact Russell-Type Modular Equations]{Exact Russell-Type Modular Equations}
\begin{abstract}
This paper provides some statistics for the coefficients of Russell-Type modular equations for the modular function, $\lambda(\tau)$. The results hold uniformly for all odd primes. They do not rely on any numerical evaluations of coefficients of $q$ expansions of $\lambda$. The method relies on an internal structure of the coefficients of $\lambda$ expressed in terms of multiplicative functions defined on integer partitions. The method may be extended to other types of modular equations.
\end{abstract}
\maketitle

\section {EXACT RUSSELL-TYPE  MODULAR EQUATIONS}

The word, exact, is used in the title in two ways. One, to distinguish this paper from others which look at the modular function,
\begin{equation*}
\lambda(\tau) = \frac{\theta^4_2(\tau)}{\theta^4_3(\tau)},
\end{equation*}
via Russell-Type modular equations. Two, to consider the coefficients in the equations expressed explicitly in terms of multiplicative functions defined on integer partitions. 

Following~\cite{Russell1887}, (though not his notation in all instances) for an odd prime, p, put
\begin{equation*}
\frac{p+1}{8} = \frac{m}{n}, \qquad (m,n) = 1
\end{equation*}
where $n$ is $1$, $2$, or $4$, and let 
\begin{align*}
 X &= \left(\lambda(\tau)\lambda\left(p\tau\right)\right)^{\frac{n}{8}} \\
Y  &= \left((1-\lambda(\tau))\left(1-\lambda\left(p\tau\right)\right)\right)^{\frac{n}{8}}.
\end{align*}
Let vectors be given as,
\begin{align*}
U &= (1, X, X^2, \cdots, X^m) \\
V &= (1, Y, Y^2, \cdots, Y^m).
\end{align*}
Then there is an $m+1 \times m+1$ matrix of integers, $A_p$, for which the triangular form,
\begin{equation*}
U A_p V' = \sum_{i,h = 0}^m a_{i,j}X^i Y^h \qquad a_{i,h} = 0 \qquad i + h > m,
\end{equation*}
vanishes on the upper half $\tau$ plane~\cite{Russell1887,ChanLiaw2000}.
Russell argues several symmetries for the coefficients, and further,
that if the constant in the form is chosen to be, $a_{0,0}= 1$, then
the $0$-th row is,
\begin{equation*}
a_{0,h} = (-1)^{h} \binom{m}{h} \qquad 0 \le h \le m.
\end{equation*}
These observations are preliminary ~\cite{Russell1887} to a gathering of as
few coefficients in the $q$ expansions of $X$ and $Y$ as
are needed to minimize the labor involved in displaying a
matrix, $A_p$.

The coefficients of the $q$ expansion of $\lambda$ are known in
various forms. One form was given by~\cite{Simons1952}. Another arises from an
expression of $\lambda/16q$ as a product-quotient. For any $i$ and
$h$,
\begin{equation*}
X^i Y^h/(2^{ni}q^{mi})
\end{equation*}
is a product-quotient, and exactly the same derivation gives the
coefficients in its $q$ expansion. The steps involve a few standard
series manipulations. The result is given in
Lemma~\ref{XYqExpansionLemma}, where an internal structure of the
coefficients of $X^i Y^h$ are displayed in terms of a
multiplicative function, $\alpha_p(k)$, whose values serve as a
measure on integer partitions. The term, exact, refers to the explicit
presentation of this internal structure.

The coefficients of $\lambda$ are partition expressions. However the
contents of the $\lambda$ coefficients are stored and utilized in
terms of a modular algorithm, the familiar complexity of partition calculations
would imply that a new expression for the $\lambda$ coefficients would
not provide a substantially new efficiency in any algorithm producing
modular equations.

When things are set up explicitly in a partition framework, however, a
highly fortuitous orthogonality occurs, which wipes out the complexity
in certain situations and allows some extremely compact formulations
for the coefficients of $A_p$ for all primes, $p$. Several general
statements of the type are given in Theorem~\ref{RowSumTheoremOne}.

Theorem~\ref{FirstRowEntriesTheorem} goes in the opposite direction. While the entries of the $0$-th row of the $A_p$ are simple functions of $p$, as are the first several moments of the first row, the entries of the first row are complicated beyond any apparent rules for simplification. This is suggested in their prime factorizations.The second element, $a_{1,1}$, in the first row, for example, is already a fair example of the irregularity encountered throughout the remarkable collection of finite matrices, $[A_p]$, where the term, irregularity, is simply a term for the mystery attached to a collection indexed on the primes. Theorem 5.2 illustrates the general form of the entries that appear for all $p$. They all depend on the same functions, which are exponential expressions in $n$ and $m$ and polynomial expressions of $\alpha_p(k)$ and polynomials in the variable, $m$, related to Stirling numbers of the second kind. These expressions display the internal structure of the entries of the $A_p$ without reducing them to arithmetic values. They consist of a main term plus terms of diminishing magnitude and explain how a sudden rise in magnitude from the $0$-th column to the $1$-th column is possible.

The form provided by $A_p$ is the descriptive, $kl-k'l'$ form of the
modular equation of $\lambda$. It is only one such form. A source of
calculations for quite a few primes is~\cite{Russell1887}. An
historical table of progress of calculations in various forms is
in~\cite{Hanna1928}.  A relatively recent introduction to the subject,
with theory, exercises, examples and references and fascinating
connections to other subjects, is~\cite{1987-borwein}. The connections
to Ramanujan, as well as to many other researchers is in an extensive
list of publications, of which we cite~\cite{Berndt2007}. One of the remarkable labor
saving devices in Russell's paper is the use of the modular equation
of $\lambda$ for $p=2$ to set up the equation for $p=13$, which led to
a method he used for a further set of primes, with equations
formulated in aesthetic combinations. ~\cite{Russell1887, Russell1890}

We do not prove that the Russell type modular equations for $\lambda$ exist, but note (the standard technique) that by showing the existence of coefficients, $a_{i,h}$, for which the linear combination, $\sum a_{i,h}X^iY^h$, is a bounded weight zero modular function, the equation is established. The only places
in the fundamental domain where the functions $X^iY^h$ are not bounded are near rational points, of the type, an odd integer over an odd integer. Near such rational points, $X$ and $Y$ may be expanded in power series similar to the ones in
this paper. For example, when $\tau$ is near $1$,
\begin{align*}
  X(\tau) &= \dfrac{1}{X(\tau/(1-\tau))} \\
  Y(\tau) &= -\dfrac{Y(\tau/(1-\tau))}{X(\tau/(1-\tau))}
\end{align*}
may be expressed as power series in
\begin{equation*}
  q_0=e^{i\pi \tau/(1-\tau)}.
\end{equation*}
These power series may be calculated in the same manner that the power series of this paper were
obtained. 
Any combination of the $X^iY^h$ that removes the negative powers in
$q_0$ for all such odd over odd rationals will necessarily
be a constant function.

There is an additional observation. The equations that cause the
negative powers in $q_0$ to disappear at one point, say at $\tau = 1$,
are equivalent to the equations that cause the negative powers
obtained at other ratios of odd integers to disappear. Therefore it is
sufficient to consider the negative $q_0$ powers for $\tau$
near~1. The two symmetries of the modular relations observed by
Russell are the equations that eliminate the infinities at the cusps,
$\tau =1$ and $\tau = \frac{1}{p}$.

Some examples of $A_p$ are nice to have to make various statements concrete. These are provided in an appendix with a program to check them with. Also in an appendix is outlined the derivation of the differential equation of $\lambda(\tau)$:
\begin{equation*}
\frac{4}{27}\frac{1 - \lambda + \lambda^2}{\lambda^2(1 - \lambda)^2}
     =  \left(\frac{2}{3}\right)^2\left( \frac{\ddot \lambda}{{\dot \lambda}^2}\right)^2
                - \left(\frac{2}{3}\right)^3\frac {\dddot \lambda}{{\dot \lambda}^3}.
\end{equation*}

\section {Some Partition Notation}

For $N\ge 1$,  and for 
$$
J = [j_1, j_2, \cdots],
$$
a finite collection of positive integers, let 

\begin{align*}
w(J) &= j_1 + 2 j_2 + 3 j_3 + \cdots \\
J[N] &= [J: w(J) = N] \\
|J| &= \sum j_k \\
J_o &= \sum_{k\  \mbox{odd}} j_k \\
J_e &= \sum_{k\  \mbox{even}} j_k.
\end{align*}

For $p$ an odd prime, let $W_p(J)$ be a function defined on the
partitions, $J \in J[N]$ (for some $N$) by  
\begin{equation*}
W_p(J) =  \prod_{k=1}^{\infty}  \frac{\alpha_p(k)^{j_k}}{j_k!} \\
\end{equation*}
where
\begin{equation*}
\alpha_p(k) = \frac{\sigma_1(k)}{k} - 3\frac{\sigma_1(\frac{k}{2})}{\frac{k}{2}} +2\frac{\sigma_1(\frac{k}{4})}{\frac{k}{4}}+ \frac{\sigma_1(\frac{k}{p})}{\frac{k}{p}} - 3\frac{\sigma_1(\frac{k}{2p})}{\frac{k}{2p}} +2\frac{\sigma_1(\frac{k}{4p})}{\frac{k}{4p}}.
\end{equation*}
This function looks a bit involved but 
\begin{equation*}
\alpha_p(k) =  \frac{\sigma_1(k)}{k} \qquad (k,2p) = 1,
\end{equation*}
and multiplicative in $k$, with simple expressions for $\alpha_p(2^a)$ and $\alpha_p(p^a)$.

The functions, $J_o$, $J_e$, $W_p(J)$, have been introduced to define
polynomials in two variables, 
\begin{align*}
b_l(u,v) &= \sum_{J[l]} (-n)^{|J|}u^{J_{o}} v^{J_{e}} W_p(J) \\
b_0(u,v) &= 1.
\end{align*} 

With,
\begin{align*}
 X &= \left(\lambda(\tau)\lambda\left(p{\tau}\right)\right)^{\frac{n}{8}} \\
Y &= \left((1-\lambda(\tau))\left(1-\lambda\left(p{\tau}\right)\right)\right)^{\frac{n}{8}},
\end{align*}
we have,
\begin{lemma}\label{XYqExpansionLemma}
\begin{equation*}
 \qquad X^i Y^h = 2^{ni}q^{mi} \sum_{l=0}^{\infty} b_l(i + 2h, i)) q^l
\end{equation*}
where the $b_l$ are the polynomials in two variables defined above.
\end{lemma}
A short initial list of some of these polynomials is given in appendix A. The last two polynomials in the appendix with asterisks are those with the second variable equal to zero, which are therefore sums over partitions with odd indices only. These are used in certain matrices in future calculations.

\begin{proof}
With
\begin{equation*}
Q = \prod_{n=1}^{\infty} \Big( 1 - q^n\Big),
\end{equation*}
we first recall the steps that take, 
\begin{equation*}
\lambda(\tau)^{\frac{n}{8}} = (\sqrt{2})^n q^{\frac{n}{8}}\frac{Q^n(\tau) Q^{2n}(4 \tau)}{Q^{3n}(2\tau)}, 
\end{equation*}
to  
\begin{equation*}
(\sqrt{2})^n q^{\frac{n}{8}}\Big( 1 + \sum_{|J|} (-n)^{|J|} \ W(J) \ q^{w(J)} \Big),
\end{equation*}
where 
\begin{align*}
W(J) &= \prod_{k=1}^{\infty} \frac{\alpha(k)^{j_k}}{j_k !} \\
\alpha(k) &= \frac{\sigma_1(k)}{k} + 2\frac{\sigma_1(k/4)}{k/4} - 3\frac{\sigma_1(k/2)}{k/2}.
\end{align*}
Thus,
\begin{align*}
 \frac{Q^n(\tau) Q^{2n}(4 \tau)}{Q^{3n}(2\tau)} &= \exp\Big[n \Big( \log[ Q(\tau)] + 2 \log[ Q(4\tau)] - 3\log [Q(2\tau)] \Big)\Big]  \\
&= \exp\left[ n\left(
      \sum_{M=1}^{\infty}(\log(1- q^M)  + 2 \log(1- q^{4M})
             -3\log(1- q^{2M}))
      \right)\right] \\
&= \exp\left[-n \sum_{N=1}^{\infty} \left(\left(\frac{\sigma_1(N)}{N} + 2\frac{\sigma_1(N/4)}{N/4} - 3\frac{\sigma_1(N/2)}{N/2}\right) q^N \right)\right]\\
&= \exp\Big[-n \sum_{N=1}^{\infty} \alpha(N) q^N\Big] \\
&= \prod_{N=1}^{\infty} e^{-n \ \alpha(N) q^N } \\
&= \prod_{N=1}^{\infty} \Big(1 + \sum_{j=1}^{\infty}\frac{(-n)^j \alpha^j(N)q^{N j}}{j!}\Big) \\
&= 1 +  \sum_{J} (-n)^{|J|} W(J) q^{w(J)} \\
&= 1 +  \sum_{l=1}^{\infty}\sum_{J[l]} (-n)^{|J|} W(J) q^{l}.
\end{align*}
For $X$ and $Y$ as above, and with,
\begin{align*}
\lambda(\tau) &=  16 q \frac{Q^8(\tau) Q^{16}(4 \tau)}{Q^{24}(2\tau)} \\
1 - \lambda(\tau) &= \frac{Q^{16}(\tau) Q^{8}(4 \tau)}{Q^{24}(2\tau)},
\end{align*}
we have, following the same steps, that 
\begin{equation*}
X^i Y^h = 2^{ni}q^{mi}\Big(1 + \sum_{l=1}^{\infty}\sum_{J[l]} (-n)^{|J|} U_{p,i,h}(J)q^l\Big),
\end{equation*}
where
\begin{align*}
U_{p,i,h}(J) &= \frac{\gamma_{p,i,h}(1)^{j_1}}{j_1!} \ \frac{\gamma_{p,i,h}(2)^{j_2}}{j_2!}\ \frac{\gamma_{p,i,h}(3)^{j_3}}{j_3!}\cdots,\\
\gamma_{p,i,h}(k) &= i \alpha_p(k) + h \beta_p(k), 
\end{align*}
and where 
\begin{align*}
\alpha_p(k) &= \frac{\sigma_1(k)}{k} - 3\frac{\sigma_1(\frac{k}{2})}{\frac{k}{2}} +2\frac{\sigma_1(\frac{k}{4})}{\frac{k}{4}}+ \frac{\sigma_1(\frac{k}{p})}{\frac{k}{p}} - 3\frac{\sigma_1(\frac{k}{2p})}{\frac{k}{2p}} +2\frac{\sigma_1(\frac{k}{4p})}{\frac{k}{4p}} \\
\beta_p(k) &= 2\frac{\sigma_1(k)}{k} - 3\frac{\sigma_1(\frac{k}{2})}{\frac{k}{2}} + \frac{\sigma_1(\frac{k}{4})}{\frac{k}{4}}+ 2\frac{\sigma_1(\frac{k}{p})}{\frac{k}{p}} - 3\frac{\sigma_1(\frac{k}{2p})}{\frac{k}{2p}} + \frac{\sigma_1(\frac{k}{4p})}{\frac{k}{4p}}.
\end{align*}
It is easy to check that when $p$ is odd, 
\begin{align*}
\beta_p(k) &= 2\alpha_p(k) \qquad  k \ \  \mbox{odd} \\
                  &= 0                  \qquad  \qquad \ k   \ \ \mbox{even}.
\end{align*}
Therefore, with $\epsilon_k =1$ if k odd, and $\epsilon_k =0$ if k even,  we can write, 
\begin{equation*}
U_{p,i,h} =   \frac{(i + 2h\epsilon_1)^{j_1}\alpha_p(1)^{j_1}}{j_1!}\ \frac{(i + 2h\epsilon_2)^{j_2}\alpha_p(2)^{j_2}}{j_2!}\ \frac{(i + 2h\epsilon_3)^{j_3}\alpha_p(3)^{j_3}}{j_3!}\cdots.
\end{equation*}
Recall that,
\begin{align*}
b_l(i+2h,i) &= \sum_{J[l]} (-n)^{|J|}  \ (i + 2h)^{J_o} \ i^{J_e} W_p(J) \\
b_0 &= 1.
\end{align*}
Thus the polynomials, $b_l$, may be used to give the $q$ expansion of $X^iY^h$, which is the lemma. 

\section {The Equations}

Substitute the expressions in Lemma~\ref{XYqExpansionLemma} into $U A_p V'$ and set the
coefficients of the powers of $q$ to zero. There results,
\begin{equation}\label{MainEquation}
0 = \sum_{i=0}^m2^{ni} \sum^{\infty}_{l=0}  \Big( \sum_{h=0}^{m-i} a_{ih}b_l(i+2h,i)\Big) q^{mi + l},
\end{equation}
the equations that determine the entries in $A_p$. 

The rows, $i=0, i=1, \cdots, i=m$ are evaluated in steps.  
Choosing $a_{0,0} = 1$,  a simple argument (Russell) gives,
\begin{equation*}
a_{0,h} = (-1)^{h} \binom{m}{h}, \qquad 0\le h\le m,
\end{equation*}
for row $0$ without solving any equations, though the methods of this
paper give the same result.  Thus, the equations for $a_{0,h}$
for $1\le h\le m$ (assuming $a_{0,0}=1$) may be expressed as
\begin{equation*}
  \left(\begin{array}{c} -1 \\ 0 \\ \vdots \\ 0\end{array}\right)
    =\left(\begin{array}{cccc}
      b_0(2,0) & b_0(4,0) & \ldots &b_0(2(m-1),0) \\
      b_1(2,0) & b_1(4,0) & \ldots &b_1(2(m-1),0) \\
               & \ddots \\
      b_{m-1}(2,0) & b_{m-1}(4,0) & \ldots &b_{m-1}(2(m-1),0) \\
      \end{array}\right)
    \left(\begin{array}{c}
      a_{0,1} \\ a_{0,2} \\ \vdots \\ a_{0,m-1}
    \end{array}\right).
\end{equation*}
The proof of lemma~\ref{Nonsingular11Determinant} on page~\pageref{Nonsingular11Determinant} may be used to prove
that the matrix
\begin{equation*}
  \left(\begin{array}{cccc}
      b_0(2,0) & b_0(4,0) & \ldots &b_0(2(m-1),0) \\
      b_1(2,0) & b_1(4,0) & \ldots &b_1(2(m-1),0) \\
               & \ddots \\
      b_{m-1}(2,0) & b_{m-1}(4,0) & \ldots &b_{m-1}(2(m-1),0) \\
  \end{array}\right)
\end{equation*}
is non-singular with determinant
\begin{equation*}
  (-2n)^{m(m-1)/2}.
\end{equation*}
This informs us that there is exactly one solution of this system of
linear equations.
Finally equation~\ref{BinomialIdentityPartA} on
page~\pageref{BinomialIdentityPartA} shows that 
\begin{equation*}
a_{0,h} = (-1)^{h} \binom{m}{h}, \qquad 0\le h\le m,
\end{equation*}
is a solution of these linear equations thus providing Russell's result.

In the $i=1$ row there are $m+1$ numbers. The last is $a_{1,m}=
0$. Besides the symmetry, $a_{i,h} = a_{h,i} $, a horizontal symmetry,
$a_{i,h} = (-1)^{m(i-1)}a_{i,m-i-h}, 0\le h \le m-i$, relates half of
the rest in pairs. So for $i = 1$ there are $\frac{m}{2}$ unknowns
among the, $a_{1,h}, 0 \le h\le m$, to find, but we do not
use the symmetry but rather solve for the $m$ values,
$a_{1,h}, 0\le h\le m-1$, via $m$ equations. The $m$ equations that
determine these values come from the equations above:
\begin{equation*}
0 = \sum_{i=0}^m2^{ni} \sum^{\infty}_{l=0}  \Big( \sum_{h=0}^{m-i} a_{ih}b_l(i+2h,i)\Big) q^{mi + l}.
\end{equation*}
The equations that come from the coefficients of $q^m, \cdots ,q^{2m-1}$ are,
\begin{align*}
0 &= 2^n\sum_{h=0}^{m-1} a_{1,h}b_0(1+2h,1) +  \sum_{h=0}^m a_{0,h}b_m(0+2h,0)  & (q^m) \\
&\cdots \\
&= 2^n\sum_{h=0}^{m-1} a_{1,h}b_{m-1}(1+2h,1) +  \sum_{h=0}^m a_{0,h}b_{2m-1}(0+2h,0) & (q^{2m-1}) 
\end{align*}
or, in matrix form,
\begin{equation*}
0 = 2^n A^{1,1} a'_1 + A^{1,0} a_0,
\end{equation*}
where
\begin{align*}
a_1 &= (a_{1,0}, a_{1,1}, \cdots, a_{1,m-1}, 0) \\
a'_1 &= (a_{1,0}, a_{1,1}, \cdots, a_{1,m-1})
\end{align*}
\begin{equation*}
A^{1,1} = [b_l(1+2h), 1)]_{l,h=0}^{m-1}        \qquad  (m\times m)
\end{equation*}
\begin{align*}
A^{1,0} &= [b_{l+m}(0+2h), 0)]_{l,h=0}^{m-1,m} \\
       &= [b_{l+m}(2h), 0)]_{l,h=0}^{m-1,m}  \qquad  (m \times m+1).
\end{align*}

\begin{lemma}\label{Nonsingular11Determinant}
The $m\times m$ matrix, $A^{1,1}$, is nonsingular with   $|A^{1,1}| = (-2n)^{\frac{m(m-1)}{2}}$.
\end{lemma}

\begin{proof}
The partition, $J\in J[l]$,  with the largest norm, $|J|$, is $j_1 = l$, so that the polynomial, $b_l$ is
\begin{equation*}
b_l(u,v) = \frac{(-nu)^l}{l!} + \mbox {lower powers of u}
\end{equation*}
Row operations reduce the determinant to a vandermonde. The factorials cancel. The power of $-2n$ that results is $1+2+\cdots + m-1$.

Next, for the $i=2$ row the equations that come from the
coefficients of $q^{2m}$, $q^{2m+1}$,~$\ldots$ $q^{3m-2}$ are,
\begin{align*}
  0 &= \begin{aligned}[t]
    2^{2n}\sum_{h=0}^{m-2} a_{2,h}b_0(2+2h,2) &+ 2^n \sum_{h=0}^{m-2} a_{1,h}b_m(1+2h,1) \\
                                          & + \sum_{h=0}^{m-2} a_{0,h}b_{2m}(0+2h,0)
    \end{aligned}&(q^{2m}) \\
&\cdots& \\
&= \begin{aligned}[t]
    2^{2n}\sum_{h=0}^{m-2} a_{2,h}b_{m-2}(2+2h,2) &+ 2^n \sum_{h=0}^{m-2} a_{1,h}b_{2m-2}(1+2h,1) \\
    &+  \sum_{h=0}^{m-2} a_{0,h}b_{3m-2}(0+2h,0)
    \end{aligned} & (q^{3m-2}) 
\end{align*}
or, in matrix form,
\begin{equation*}
0 = 2^{2n} A^{2,2} a'_2 + 2^n A^{2,1} a'_1 + A^{2,0} a_0,
\end{equation*}
\begin{align*}
 A^{2,2}  &= \Big[b_l(2+2h, 2)\Big]_{l,h = 0}^{m-2} \qquad  \qquad (m-1 \times m-1) \\
 A^{2,1}  &= \Big[b_{l+m}(1+2h,1)\Big]_{l,h=0}^{m-2,m-1} \qquad   \qquad (m-1 \times  m) \\
 A^{2,0}  &= \Big[b_{l+2m}(2h,0)\Big]_{l,h=0}^{m-2,m} \qquad   \qquad \ (m-1 \times m+1). 
\end{align*}
The initial piece, $a'_2$, of the second row is given by the solution of 
\begin{equation*}
2^{2n}A^{2,2}a'_2 + 2^n A^{2,1}a'_1 + A^{2,0} a_0 = 0,
\end{equation*}
for the $m-1$ vector,
\begin{equation*}
a'_2 = \Big( a_{2,0}, a_{2,1}, \cdots, a_{2, m-2}\Big),
\end{equation*}
\end{proof}

\begin{lemma}
The $m-1 \times m-1$ matrix, $A^{2,2}$, is nonsingular with determinant,  $|A^{2,2}| = (-2n)^{\frac{(m-1)(m-2)}{2}}$
\end{lemma}
\begin{proof}
Same as Lemma~\ref{Nonsingular11Determinant}.
\end{proof}

The general case is summarized as follows.  

  For each, $i$, $1 \le i \le m$, there are $i+1$ matrices, $A^{i,r}, 0\le r\le i$ that specify the $i$-th row of $A_p$, where, $a'_i$ (being the first $m+1-i$ elements of the $i$-th row, $a_i$, of $A_p$) is the solution of

\begin{equation*}
2^{n i} A^{i,i} a'_i + 2^{n(i-1)} A^{i,i-1} a'_{i-1} + \cdots + A^{i,0} a_0 =0.
\end{equation*}

 The vectors with lower indices have already been solved for. The matrix, $A^{i,r}$, has rows indexed on, $0\le l \le m-i$, and  columns indexed on, $0 \le h \le m-r$, defined via the polynomials, $b_N(u,v)$, as

\begin{equation*}
A^{i,r} = [b_{l + m(i-r)}(r+2h), r)] \qquad 0\le l \le m-i,\  0\le h \le m -r.
\end{equation*}

 The $m+i-1 \times m+i-1$ matrix, $A^{i,i}$, is nonsingular, with determinant,

\begin{equation*}
|A^{i,i}| = (-2n)^{\frac{(m+1-i)(m-i)}{2}}.
\end{equation*}

\section {Explicit Statements Row $i=1$}

From what has been said above, the coordinates of, $a'_1$, appear in the solution of,
\begin{equation*}
0 = 2^n A^{1,1} a'_1 + A^{1,0} a_0.
\end{equation*}
To make this explicit we first evaluate the components of the vector, $-A^{1,0}a_0$, the $l$-th component, $0\le l \le m-1$, being
\begin{equation*}
\Big(-A^{1,0}a_0\Big)_l = \sum_{h=0}^{m}b_{l+m}(2h,0) (-1)^{h-1}\binom{m}{h}.
\end{equation*}
This is where we get very lucky. The polynomial, 
\begin{equation*}
b_{l+m}(2h, 0) =  \sum_{J[l+m]} (-n)^{|J|}  (2h)^{J_{o}}0^{J_{e}} W_p(J),
\end{equation*}
is a sum of powers of $h$. Since
\begin{align}
\sum_{h=0}^{m} h^N (-1)^{h-1}\binom{m}{h} &= 0 &0 \le  N < m, \label{BinomialIdentityPartA} \\
&= (-1)^{m-1}P_{N-m} (m) N!                   &  m \le N, \nonumber
\end{align}
where $P_{N-m}(m)$ is a polynomial of degree $N-m$, we can express the $l$-th component,
\begin{equation*}
(-A^{1,0}a_0)_l,
\end{equation*}  
as
\begin{equation*}
  \sum_{d=m}^{l+m} (-1)^{d+m-1}n^d 2^d P_{d-m}(m) \sum_{|J|=d} \alpha_p(1)^{j_1}\cdots \alpha_p(l+m)^{j_{l+m}}\ \frac{d!}{j_1! j_2! \cdots j_{l+m}!},
\end{equation*}
where the partitions involved consist of $j_k$ with only $k$ odd.

Note that the outer sum starts at $d = m$ because the only partitions of $l+m$
that contribute non zero terms to the inner sum are those for which $|J|$ is no smaller than $m$, with the condition that the subscripts are odd always in force. (This explains the asterisks on certain polynomials in an appendix). This allows for some quite compact statements in the next section.

Note: The inner sum on partitions, $J$, for which $|J| = d$, is equal to the terms in the expansion of
\begin{equation*}
\Big(\alpha_p(1) + \alpha_p(2) + \cdots \alpha_p(l+m)\Big)^d
\end{equation*}
for which
\begin{equation*}
 j_1 + 2 j_2 + \cdots (l+m) j_{l+m} = l+m,
\end{equation*}
and for which $j_k$ = 0 for all even $k$. The polynomials may be calculated from the recursion,   
\begin{equation*}
d P_{d-m}(m) = m\Big( P_{d-m-1}(m) + P_{d-m}(m-1) \Big) \qquad P_0(m) = 1.
\end{equation*}
A table of the few polynomials used in this paper is in the appendix. A binomial expression satisfied by the $P_N$ is provided but not used in this paper.

\section {Theorems}

\begin{theorem}\label{RowSumTheoremOne}
For all odd primes, $p$,   
\begin{align*}
\sum^{m}_{h=0}a_{1,h}  &= -n^m 2^{m-n} \\
\sum^{m}_{h=0}(1 + 2h)a_{1,h}  &= -n^m 2^{m+1-n} P_1(m) \\
\sum^{m}_{h=0}(1+2h)^2 a_{1,h}  &=
\begin{aligned}[t]
     -n^m &2^{m+3-n} P_2(m) \\
&+  n^{m-1} 2^{m+1-n}  \alpha_p (2) \\
     &- n^{m-2} 2^{m+1-n}  m  \alpha_p (3).
\end{aligned}
\end{align*}
\end{theorem}
\begin{proof} We first express  $\Big(-A^{1,0}a_0\Big)_l$ explicitly in terms of $d = l+m, d = l+m-1$, and so on. 

Thus, for $d = l+m$, we have $J = [j_1 = l+m]$, which contributes a term,
$$
(-1)^{l-1}n^{l+m}2^{l+m} P_l(m) \frac{(l+m)!}{(l+m)!}\alpha_p(1)^{l+m} = (-1)^{l-1}n^{l+m}2^{l+m}P_l(m),
$$
as $\alpha_p(1) = 1$. 

For $d = l+m-1$, the partition is $J = [j_1 = l+m-2, j_2 =1]$, and there is no contribution due to $j_2 = 1$.  

For $d = l+m-2$, the partition is $J = [j_1 = l+m-3, j_3 =1]$, and the contribution is
\begin{align*}
(-1)^{l-3}n^{l+m-2}2^{l+m-2}&P_{l-2}(m) \frac{(l+m-2)!}{(l+m-3)!}\alpha_p(1)^{l+m-3}\alpha_p(3) \\
            & = (-1)^{l-3}n^{l+m-2}2^{l+m-2}P_{l-2}(m)(l+m-2)\alpha_p(3).
\end{align*}

For $d = l+m-3$, the contribution is $0$ by a very similar argument as that for $d = l+m-1.$

For $d = l+m-4$, there are two partitions, $J = [j_1 = l+m-6, j_3 =2]$, and, $J = [j_1 = l+m-5, j_5 =1]$ , and the contribution, after simplification, is
\begin{align*}
(-1)^{l-5}\Big(n^{l+m-4}2^{l+m-4}&P_{l-4}(m)(l+m-4)(l+m-5)\frac{\alpha^2_p(3)}{2!} \\
                             & + n^{l+m-4}2^{l+m-4}P_{l-4}(m)(l+m-4)\alpha_p(5)\Big).
\end{align*}

Altogether we obtain,

\begin{align*}
\Big(-A^{1,0}a_0\Big)_l &= (-1)^{l-1} \Big( 2^{l+m} n^{l+m} P_l(m) \\
&+ 2^{l+m-2} n^{l+m-2} P_{l-2}(m)(l+m-2)\alpha_p(3) \\
&+ 2^{l+m-4} n^{l+m-4} P_{l-4}(m)(l+m-4)(l+m-5)\frac{\alpha^2_p(3)}{2!} \\
&+ 2^{l+m-4} n^{l+m-4} P_{l-4}(m)(l+m-4)\alpha_p(5) + \cdots \Big). 
\end{align*}
This gives
\begin{align*}
\Big(-2^{-n}A^{1,0}a_0\Big)_0 &= -2^{m-n} n^m P_0(m) \\  
\Big(-2^{-n}A^{1,0}a_0\Big)_1 &= 2^{1+m-n}n^{1+m} P_1(m) \\
\Big(-2^{-n}A^{1,0}a_0\Big)_2 &= -\Big(2^{2+m-n}n^{2+m} P_2(m) + 2^{m-n} n^m P_0(m) m \alpha_p(3)\Big).
\end{align*}
The polynomials, $P_0, P_1, P_2$ are available in the appendix for explicit functions of $m$. The polynomials,
\begin{align*}
b_0(1+2h,1) &= 1 \\
b_1(1+2h,1) &= -n(1+2h) \\
b_2(1+2h,1) &= \frac{n^2}{2}(1+2h)^2 - n\alpha_p(2)
\end{align*}
provide the $0$-th, $1$-st, and $2$-nd rows of $A^{1,1}$ for $0\le h\le m-1$, and therefore the first three components of $A^{1,1}a'_1$, which are (since $a_{1,m} = 0)$,
\begin{align*}
&\sum^{m}_{h=0}a_{1,h}  \\
-n&\sum^{m}_{h=0}(1+2h) a_{1,h}   \\
 \frac{n^2}{2}&\sum (1+2h)^2a_{1,h} - n\alpha_p(2)\sum^{m}_{h=0}a_{1,h} 
\end{align*}

The three statements of the theorem follow. Higher statistics are readily available, but these three provide a sense of the increasing complexity.
\end{proof}

\begin{theorem}\label{FirstRowEntriesTheorem}

\noindent Part 1. 

\noindent The first three entries of the first row of $A_5 (m=3, n=4)$, of $A_{11}(m=3, n=2)$, and of $A_{23}(m=3, n=1)$ are given by the expressions,
\begin{align*}
a_{1,0} &= 2^{m-n-3}n^{m-2}\Big(2n\alpha_p(2) + 2m\alpha_p(3) + n^2(15 - 16 P_1(m) + 8 P_2(m) \Big) \\
a_{1,1} &= 2^{m-n-3}n^{m-2}\Big(-4n\alpha_p(2) - 4m\alpha_p(3) - 2 n^2(5 - 12 P_1(m) + 8 P_2(m) \Big) \\
a_{1,2} &= 2^{m-n-3}n^{m-2}\Big(2n\alpha_p(2) + 2m\alpha_p(3) + n^2(3 - 8 P_1(m) + 8 P_2(m) \Big) \\
\end{align*}
\noindent Part 2. The first two elements of the second row of the entries of $A_5 (m=3, n=4)$, of $A_{11}(m=3, n=2)$, and of $A_{23}(m=3, n=1)$ are given by the expressions,
\begin{align*}
a_{2,0} = &(1/3)n^m 2^{m-2n-3} \Big((464P_2(m) - 384P_1(m) + 135) n^4 \\
& -4(24P_2(m) - 29)n^3\alpha_p(2) + n^2(-12\alpha^2_p(2) + 4(48P_2(m) + 29m)\alpha_p(3)) \\
& + n(-24(m - 2)\alpha_p(6) - 24\alpha_p(4)) + 48m \alpha^2_p(3)  \Big)\\
&-2^{2m-2n+1} n^{2m+1}P_4(m) -n^{2m-1} 2^{2m-1}P_2(m) (2m-1)\alpha_p(3) \\
       &- n^{2m-2n-3} 2^{2m-3} \Big((2m-3)(2m-4)\frac{\alpha^2_p(3)}{2!}+ (2m-3)\alpha_p(5)  \Big)
\end{align*}

$a_{2,1} = -a_{2,0}$

\noindent Part 3. When $m = 3$,

$$
\left|\begin{array}{ccc}
b_0(3,1) & b_0(5,1) & \Big(A^{1,0}a_0\Big)_0  \\
b_1(3,1) & b_1(5,1) & \Big(A^{1,0}a_0\Big)_1  \\
b_2(3,1) & b_2(5,1) & \Big(A^{1,0}a_0\Big)_2
\end{array}\right| = m (-2n)^3  2^{n} 
$$

\end{theorem} 

\noindent Proof of part 1.   We apply Cramer's Rule to, 
\begin{equation*}
A^{1,1} a'_1 =  -2^{-n}A^{1,0} a_0,
\end{equation*}
where
\begin{equation*}
  A^{1,1} = 
\left(\begin{array}{ccc}
1 & 1 & 1  \\
-n & -3n & -5n   \\
\frac{n^2}{2} -n \alpha_p(2) & \frac{9n^2}{2} -n \alpha_p(2) & \frac{25 n^2}{2} -n \alpha_p(2)  
\end{array}\right) 
\end{equation*}
and where
\begin{align*}
-\Big(2^{-n}A^{1,0}a_0\Big)_0 &= -n^m 2^{m-n} \\
-\Big(2^{-n}A^{1,0}a_0\Big)_1 &=  n^{m+1} 2^{m+1-n} P_1(m)\\
-\Big(2^{-n}A^{1,0}a_0\Big)_2 &=  -n^{m+2} 2^{m+2-n}P_2(m) - n^m 2^{m-n} m \ \alpha_p (3), 
\end{align*}
to get the results given above.  

\noindent Proof of Part 2.

\noindent The linear system satisfied by $a_2'$ is      
 
\begin{align*}
2^{2n}A^{2,2} a'_2 &= -2^{n} A^{2,1} a'_1 - A^{2,0} a_0 \\
&= A^{2,1}\Big(A^{1,1}\Big)^{-1} \Big(A^{1,0}a_0\Big)  - A^{2,0} a_0.
\end{align*}
We calculate the components of the vector, $-A^{2,0} a_0$, the $l$-th component being,
\begin{align*}
-\Big(A^{2,0}a_0\Big)_l = \sum_{d=m}^{l+2m} \Big((-1)^{d+m-1}&n^d 2^d P_{d-m}(m) \\
                   &\sum_{|J|=d} \alpha_p(1)^{j_1}\cdots \alpha_p(l+m)^{j_{l+m}}\ \frac{d!}{j_1! j_2! \cdots j_{l+2m}!}\Big).
\end{align*}
How much needs to be calculated for the $i = 2$ row of $m=3$? The matrix, $A^{2,0}$, has dimensions $2\times 4$, so that $l$ runs from $l = 0$ to $l = 1$. For this largest value of $l = 1$ we will have a largest, $d = l + 2m = 1 + 2\times3 = 7$ and a smallest $d= 7 - 2x \ge 3$, so that $x = 2$ and therefore we need partitions of $l + 2m, l + 2m -2, l + 2m -4$. We begin at the top. 
 
Thus, for $d = l+2m$, we have $J = [j_1 = l+2m]$, which contributes one term,
\begin{align*}
(-1)^{l+m-1}n^{l+2m}2^{l+2m}&P_{l+m}(m) \frac{(l+2m)!}{(l+2m)!}\alpha_p(1)^{l+2m} \\
                        &= (-1)^{l+m-1}n^{l+2m}2^{l+2m}P_{l+m}(m),
\end{align*}
as $\alpha_p(1) = 1$. (We use $(-1)^{2m} = 1$). For $d = l+2m-1$, the partition is $J = [j_1 = l+2m-2, j_2 =1]$, and there is no contribution due to $j_2 = 1$ (only odd indices survive.) For $d = l+2m-2$, the partition is $J = [j_1 = l+2m-3, j_3 =1]$, and the contribution is
\begin{align*}
&(-1)^{l+m-1}n^{l+2m-2}2^{l+2m-2} P_{l+m-2}(m) \frac{(l+2m-2)!}{(l+2m-3)!}\alpha_p(1)^{l+2m-3}\alpha_p(3) =\\ &(-1)^{l+m-1}n^{l+2m-2}2^{l+2m-2}P_{l+m-2}(m)(l+2m-2)\alpha_p(3) + \cdots.
\end{align*} 
For $d = l+2m-3$, there is no contribution. For $d = l+2m-4$, the partitions are $J = [j_1 = l+2m-5, j_5 =1]$ and $J = [j_1 = l+2m-6, j_3 =2]$, and the sum of the two terms that contribute is,
\begin{align*}
(-1)^{l+m-1}&n^{l+2m-4}2^{l+2m-4} P_{l+m-4}(m) \\
          &\Big( (l+2m-4)\alpha_p(5) + \frac{(l+2m-4)(l+2m-5)}{2!}\alpha^2_p(3)  \Big).
\end{align*}

The three expressions (for d = l+2m, d = l+ 2m - 2, d = l + 2m -4) are sufficient to give the $l = 0, l = 1$ components of the vector, $-A^{2,0}a_0$, when $m=3$. (Recall the matrix, $A^{2,0}$, is $2\times 4$).

They are, (using $(-1)^{m-1} = (-1)^{3-1} = 1$)
\begin{align*}
  -\Big(A^{2,0}&a_0\Big)_0 \\
  & = \begin{aligned}[t]
&n^{2m} 2^{2m} P_m(m) \\
              &+ n^{2m-2} 2^{2m -2}P_{m-2}(m)\ (2m - 2) \alpha_p(3) \\
              &+ n^{2m-4} 2^{2m -4}P_{m-4}(m) \Big(\frac{(2m - 4)(2m-5)}{2!} \alpha^2_p(3) +  (2m-4) \alpha_p(5)\Big)
    \end{aligned} \\
-\Big(A^{2,0}&a_0\Big)_1 \\
& = 
     \begin{aligned}[t]
       - \Big(&n^{2m+1} 2^{2m +1}P_{m+1}(m)  \\
            &+n^{2m-1} 2^{2m-1}P_{m-1}(m) (2m-1)\alpha_p(3) \\
       &+ n^{2m-3} 2^{2m-3}P_{m-3}(m) \Big((2m-3)(2m-4)\frac{\alpha^2_p(3)}{2!}+ (2m-3)\alpha_p(5)  \Big)  \\
       &+ n^{2m-5} 2^{2m-5}P_{m-5}(m) \Big( (2m-5)\alpha_p(7) \\ 
 &  +(2m-5)(2m-6)\alpha_p(3)\alpha_p(5) \\  
 &+ \frac{(2m-5)(2m-6)(2m-7)}{3!}\alpha^3_p(3) \Big)  \\
       \end{aligned}\\
\end{align*}
We now keep as much of the terms that contribute in the two components of the column vector for $m=3$: 
\begin{align*}
-(A^{2,0}a_0)_0 &= 2^{2m} n^{2m}P_3(m) +  n^{2m-2} 2^{2m -2}P_1(m)\ (2m - 2) \alpha_p(3) \\
-(A^{2,0}a_0)_1 &= -2^{2m+1} n^{2m+1}P_4(m) -n^{2m-1} 2^{2m-1}P_2(m) (2m-1)\alpha_p(3) \\
       &- n^{2m-3} 2^{2m-3} \Big((2m-3)(2m-4)\frac{\alpha^2_p(3)}{2!}+ (2m-3)\alpha_p(5)  \Big).  
\end{align*}
Next, we have already calculated the 3-vector, $A^{1,0}a_0$, which is given above, with components, 
\begin{align*}
A^{1,0}a_0\Big|_0 &= n^m 2^{m} \\
A^{1,0}a_0\Big|_1 &= - n^{m+1} 2^{m+1} P_1(m)\\
A^{1,0}a_0\Big|_2 &=  n^{m+2} 2^{m+2}P_2(m) + n^m 2^{m} m \ \alpha_p (3), 
\end{align*}
and we need the $2\times 3$ matrix, $A^{2,1}\Big(A^{1,1}\Big)^{-1},$ which acts on it. Then calculate that, $A^{2,1}\Big(A^{1,1}\Big)^{-1}$ is

\begin{equation*}
\left(\begin{array}{ccc}
  -\frac{n^2}{2} (5n + 6 \alpha_p(2))&
  \left(\begin{aligned}
    -\frac{23 n^2}{6} &- n \alpha_p(2)\\
    & + \alpha_p(3)
  \end{aligned}\right)& -3n  \\
 \left(\begin{aligned}
   \frac{45n^4}{8} &- \frac{n^2}{2}\alpha_p^2(2) \\
   &+ \alpha_p(2)\frac{29 n^3}{6} \\
   & + 2 n \alpha_p(2)\alpha_p(3) \\
      & -n\alpha_p(4)
 \end{aligned}\right)
& 8n^3 &  
 \left(\begin{aligned}
   \frac{29 n^2}{6} &- n \alpha_p(2)\\
                    & + 2\alpha_p(3)
  \end{aligned}\right)
\end{array}\right). 
\end{equation*}

The coordinates of $A^{2,1}\Big(A^{1,1}\Big)^{-1} A^{1,0}a_0$ are, then,

\begin{align*}
A^{2,1}(A^{1,1})^{-1}&A^{1,0}a_0\Big |_0 \\
   & = \begin{aligned}[t]
      (1/3)n^{m+1}2^{m-1}\Big((- 15 &+ 46 P_1(m)-72 P_2(m)  )n^2 \\
          &+ 6( 2 P_1(m) -3) n \alpha_p(2) \\
          &- 6 (2 P_1(m) + 3m)\alpha_p(3)\Big)
  \end{aligned} \\
A^{2,1}(A^{1,1})^{-1}&A^{1,0}a_0\Big |_1 \\
 & = \begin{aligned}[t]
      (1/3)n^m 2^{m-3} \Big(&(464P_2(m) - 384P_1(m) + 135) n^4 \\
      & -4(24P_2(m) - 29)n^3\alpha_p(2) \\
      & + n^2(-12\alpha^2_p(2) \\
      & + 4(48P_2(m) + 29m)\alpha_p(3)) \\
      & + n(-24(m - 2)\alpha_p(6) - 24\alpha_p(4)) \\
      & + 48m \alpha^2_p(3)\Big). 
  \end{aligned}
\end{align*}
When $m=3$ the first coordinate of the sum,
$$
A^{2,1}(A^{1,1})^{-1}A^{1,0}a_0 - A^{2,0}a_0
$$
is zero. The action of 

$$
(A^{2,2})^{-1}=\left(\begin{array}{cc}
2 & 1/2n  \\
-1 & -1/2n  
\end{array}\right) 
$$
on $A^{2,1}(A^{1,1})^{-1}A^{1,0}a_0 - A^{2,0}a_0$ is therefore,
$$
\frac{1}{2n}(C, -C),
$$  
where $C$ is the second coordinate, 
$$
\Big(A^{2,1}(A^{1,1})^{-1}A^{1,0}a_0 - A^{2,0}a_0\Big) \Big |_1.
$$
Therefore,
\begin{align*}
2^{2n}a_{2,0} &= C \\
2^{2n}a_{2,1} &= -C,
\end{align*}
which is explicitly written in the statement of Part 2 of the Theorem~\ref{FirstRowEntriesTheorem}.

\noindent Note in the formulas that $\alpha_p(5) = \frac{11}{5}$ when $p = 5, n=4, m= 3$ but $\alpha_p(5) = \frac{6}{5}$ when $p = 11, n=2, m= 3$ and $\alpha_p(5) = \frac{6}{5}$ when $p = 23, n=1, m= 3$.

\noindent Proof of Part 3. Knowing that $a_{1,0} = -m$, and knowing the $0$-th column ($l = 0$) of,

$$
\Big [b_l(1+ 2h,1]\Big ]_{l,h =0}^{m-1,m},
$$
we can transform the linear system,
$$
0 = 2^n \Big [b_l(1+ 2h,1]\Big ]_{l,h =0}^{m-1,m}a_1 +  \Big [b_{l+m}(2h,0]\Big ]_{l,h =0}^{m-1,m}a_0,
$$
to
$$
0 = 2^n \Big [b_l(1+ 2h,1]\Big ]_{l,h =0,1}^{m-1,m}a''_1 +  2^n(-m)\Big[b_l(1,1)\Big ]_{l=0}^{m-1} +    \Big [b_{l+m}(2h,0]\Big ]_{l,h =0}^{m-1,m}a_0,
$$
where $a''_1$ is the column vector, $\Big[a_{1,h}\Big ]_{h=1}^{m}$. Then substitute the vector sum,
$$
m \Big[b_l(1,1)\Big ]_{l=0}^{m-1} - 2^{-n} \Big [b_{l+m}(2h,0]\Big ]_{l,h =0}^{m-1,m}a_0 
$$
into the third (and final) column of the square matrix,
$$
\Big [b_l(1+ 2h,1]\Big ]_{l,h =0,1}^{m-1,m},
$$
to obtain, when $m = 3$,
$$
a_{1,m} = 0 = m\left|\begin{array}{ccc}
b_0(3,1)& b_0(5,1) & b_0(1,1)  \\
b_1(3,1)& b_1(5,1) & b_1(1,1)  \\
b_2(3,1)& b_2(5,1) & b_2(1,1)  
\end{array}\right| - 2^{-n} \left|\begin{array}{ccc}
b_0(3,1) & b_0(5,1) & \Big(A^{1,0}a_0\Big)_0  \\
b_1(3,1) & b_1(5,1) & \Big(A^{1,0}a_0\Big)_1  \\
b_2(3,1) & b_2(5,1) & \Big(A^{1,0}a_0\Big)_2
\end{array}\right|
$$
From Lemma~\ref{XYqExpansionLemma}, 
$$
\left|\begin{array}{ccc}
b_0(3,1)& b_0(5,1) & b_0(1,1)  \\
b_1(3,1)& b_1(5,1) & b_1(1,1)  \\
b_2(3,1)& b_2(5,1) & b_2(1,1)  
\end{array}\right| = (-2n)^3, 
$$
and we get Part 3 of the theorem. We note that if we had established the identity of Part 3 directly we would have proved that $a_{1,m} = 0$ in the case $m = 3$.

\end{proof}

The expressions illustrate the complexity of the entries of  $A_5$, $A_{11}$, $A_{23}$ before the parameters are  evaluated numerically. These are in the appendix for comparison. Part 3 of Theorem~\ref{FirstRowEntriesTheorem} indicates in the simplest instance that the relations, $a_{l,h} = 0$ for $l + h > m$, are non trivial.

\appendix
\section{}

A Table of $b_N(u,v)$

\begin{align*}
b_0(u,v) &= 1 \\
b_1(u,v) &= -n u\\
b_2(u, v) &= n^2 u^2\frac{1}{2!} - nv \ \alpha_p(2) \\ 
b_3(u,v) &= -n^3 u^3\frac{1}{3!} + n^2 u v \ \alpha_p(2) - n u  \ \alpha_p(3) \\
b_4(u,v) &= n^4 u^4\frac{1}{4!} -n^3 u^2v\frac{1}{2!} \ \alpha_p(2) + n^2 u v \ \alpha_p(3) + n^2 v^2\frac{\alpha^2_p(2)}{2!} - n v  \ \alpha_p(4) \\
b_5(u,v) &= -n^5 u^5\frac{1}{5!} + n^4 u^3 v \frac{1}{3!} \alpha_p(2) - n^3 u^3 \frac{1}{2!} \alpha_p(3)- n^3 u v^2  \frac{\alpha^2_p(2)}{2!}\\
&+ n^2 u v \ \alpha_p(4) + n^2 u v \ \alpha_p(3)\alpha_p(2)  - n u  \ \alpha_p(5)  \\
b^{*}_6(u,v) &= n^6 u^6\frac{1}{6!} + n^4 u^4 \frac{\alpha_p(3)}{3!} + n^2 u^2 \frac{\alpha^2_p(3)}{2!} + n^2 u^2 \alpha_p(5)\\
b^{*}_7(u,v) &= -n^7 u^7\frac{1}{7!} - n^5 u^5 \frac{\alpha_p(3)}{4!} - n^3 u^3 \frac{\alpha_p(5)}{2!} - n^3 u^3 \frac{\alpha^2_p(3)}{2!} - n\alpha_p(7)
\end{align*}

\section {}

\begin{equation*}
 A_5 = 
\left(\begin{array}{cccc}
1 & -3 & 3 & -1 \\
-3 & -26 & -3 & 0  \\
3 & -3 & 0 & 0 \\
-1 & 0 & 0 & 0
\end{array}\right) \qquad  n=4 \ m=3
\end{equation*}

\begin{equation*} A_{11} = 
\left(\begin{array}{cccc}
1 & -3 & 3 & -1 \\
-3 & -10 & -3 & 0  \\
3 & -3 & 0 & 0 \\
-1 & 0 & 0 & 0
\end{array}\right) \qquad n=2 \ m=3
\end{equation*}

\begin{equation*}
  A_{23} = 
\left(\begin{array}{cccc}
1 & -3 & 3 & -1 \\
-3 & 2 & -3 & 0  \\
3 & -3 & 0 & 0 \\
-1 & 0 & 0 & 0
\end{array}\right) \qquad n=1 \ m=3
\end{equation*}

\begin{equation*}
A_{31} = 
\left(\begin{array}{ccccc}
1 & -4 & 6 & -4 & 1 \\
-4 & 0 & 0 & -4 & 0  \\
6 & 0 & 6 & 0 & 0 \\
 -4 & -4 & 0 & 0 & 0 \\
1 & 0 & 0 & 0 & 0
\end{array}\right) \qquad n = 1 \ m = 4
\end{equation*}

\begin{equation*}
 A_{19} = 
\left(\begin{array}{cccccc}
1 & -5 & 10 & -10 & 5 & -1 \\
-5 & -92 & -62 & -92 & -5 & 0  \\
10 & -62 & 62 & -10 & 0 & 0 \\
-10 & -92 & -10 & 0 & 0 & 0 \\
5 & -5 & 0 & 0 & 0 & 0 \\
-1 & 0& 0 & 0 & 0 & 0, 
\end{array}\right) \qquad n = 2 \ m = 5
\end{equation*}

\begin{equation*}
A_{47} = 
\left(\begin{array}{ccccccc}
1 & -6 & 15 & -20 & 15 & -6 & 1\\
-6 & -10 & 0 & 0 & -10 & -6 & 0\\
15 & 0 & -14 & 0 & 15 & 0 & 0\\
-20 & 0 & 0 & -20 & 0 & 0 & 0\\
15 &-10 & 15 & 0 & 0 & 0 & 0\\
-6 & -6 & 0 & 0 & 0 & 0 &  0\\
 1 & 0 & 0 & 0 & 0 & 0 & 0 \\
\end{array}\right) \qquad n = 1 \ m = 6
\end{equation*}

\begin{align*}
A_{13} &= 
\left(\begin{array}{cccccccc}
1 & -7 & 21 & -35 & 35 & -21 &7 & -1 \\
-7 & -3318 & -31721 & -60980 & -31721 & -3318 & -7 & 0 \\
21 & -31721 & -11438 & 11438 & 31721 & -21 & 0 & 0 \\
-35 & -60980 & 11438 & -60980 & -35 & 0 & 0 & 0 \\
35 & -31721 & 31721 & -35 & 0 & 0 & 0 & 0 \\
-21 & -3318 & -21 & 0 & 0 & 0 &  0 & 0 \\
7 & -7 & 0 & 0 & 0 & 0 & 0 & 0 \\
-1 & 0 & 0 & 0 & 0 & 0 & 0 & 0 
\end{array}\right)\\
 & n = 4,  m = 7
\end{align*}

\section{}

\begin{equation*}
P_0(m) = 1
\end{equation*}

\begin{equation*}
P_1(m) = \frac{m}{2}
\end{equation*}

\begin{equation*}
P_2(m) = \frac{m(3m+1)}{24}
\end{equation*}

\begin{equation*}
P_3(m) = \frac{m^2(m+1)}{48}
\end{equation*}

\begin{equation*}
P_4(m) = \frac{m(-2+ 5m + 30m^2 + 15 m^3)}{8 (6!)}
\end{equation*}

The polynomials may be written,
\begin{equation*}
P_{N-m}(m) = \sum_{r=0}^{N-m} c_{N-m,r} \binom{m}{r}
\end{equation*}

\begin{align*}
c_{0,0} &= 1 \\
c_{N-m, r} &= \frac{r}{N-m + r}\Big( c_{N-m -1, r-1} + c_{N-m -1, r} \Big) \qquad 0\le r \le N-m \\
c_{N-m-1, N-m} &= 0  
\end{align*}

\section{}

\begin{equation*}
\alpha_p(1)  =  1
\end{equation*}

\begin{equation*}
\alpha_p(2)  =  -\frac{3}{2} \qquad p \neq 2
\end{equation*}

\begin{equation*}
\alpha_p(3)  =  \frac{4}{3} \qquad p \neq 2, p\neq 3
\end{equation*}

\begin{equation*}
\alpha_p(4)  =  -\frac{3}{4} \qquad p \neq 2
\end{equation*}

\begin{equation*}
\alpha_p(5)  =  \frac{6}{5} \qquad p \neq 2, p\neq 5 \qquad (= \frac{11}{5} \qquad p = 5)
\end{equation*}

\section{Differential equation for $\lambda$}

The modular lambda function satisfies an elegant non-linear
differential equation, 
\begin{equation*}
        2\dfrac{f'''(\tau)}{f'(\tau)^3}
             -3\dfrac{f''(\tau)^2}{f'(\tau)^4}
             =-\dfrac{f(\tau)^2-f(\tau)+1}{f(\tau)^2(1-f(\tau))^2}.
\end{equation*}
An interesting fact about this differential equation is that it has a
variety of solutions that are algebraically related to the lambda
function:
\begin{align*}
  f(\tau) &= \lambda(p\tau) \\
  f(\tau) &= \lambda\left(\dfrac{\tau+2k}{p}\right) & 0\le k < p
\end{align*}

The differential equation for $f = \lambda$ is very easily verified.  It
follows from the fact that the differential operator
\begin{equation*}
        f(\tau)\rightarrow
          2\dfrac{f'''(\tau)}{f'(\tau)^3}
             -3\dfrac{f''(\tau)^2}{f'(\tau)^4}
\end{equation*}
is invariant under linear fractional transformations.  More precisely,
if $L_0$ is a fractional linear transformation,
\begin{equation*}
  L_0(\tau) = \dfrac{a\tau+b}{c\tau+d}
\end{equation*}
where
\begin{equation*}
  ad - bc > 0
\end{equation*}
and
\begin{equation*}
  f_1(\tau) = f_0(L_0(\tau))
\end{equation*}
then
\begin{equation*}
        2\dfrac{f_1'''(\tau)}{f_1'(\tau)^3}
             -3\dfrac{f_1''(\tau)^2}{f_1'(\tau)^4}
        =2\dfrac{f_0'''(L_0(\tau))}{f_0'(L_0(\tau))^3}
             -3\dfrac{f_0''(L_0(\tau))^2}{f_0'(L_0(\tau))^4}.
\end{equation*}
This assertion may be quickly verified through direct computation.  It
follows from this that if
\begin{equation*}
  f(\tau) = \lambda(\tau).
\end{equation*}
then 
\begin{equation*}
        2\dfrac{f'''(\tau)}{f'(\tau)^3}
             -3\dfrac{f''(\tau)^2}{f'(\tau)^4}
\end{equation*}
is a weight zero modular form with singularities at the cusps.  This
weight zero modular form must be a rational function of
$\lambda(\tau)$ and a direct computation shows that this rational
function is
\begin{equation*}
     -\dfrac{\lambda(\tau)^2-\lambda(\tau)+1}{\lambda(\tau)^2(1-\lambda(\tau))^2}.
\end{equation*}
This differential equation differs from the equation on the Wolfram function site.

\bibliography{Modular}{}

\end{document}